\documentclass[12pt]{amsart}

\usepackage{amssymb}
\usepackage{mathrsfs}

\newcommand{\R}{\mathbf{R}}
\newcommand{\Z}{\mathbf{Z}}
\newcommand{\T}{\mathbf{T}}
\newcommand{\seq}[1]{\mathtt{#1}}

\newcommand{\PSL}{\operatorname{PSL}}
\newtheorem{thm}{Theorem}[section]
\newtheorem{fact}[thm]{Fact}
\newtheorem{prop}[thm]{Proposition}
\newtheorem{example}[thm]{Example}
\newtheorem{problem}[thm]{Problem}

\title[Amenable equivalence relations]
{A brief introduction to amenable equivalence relations}

\keywords{amenable, equivalence relation, free group, hyperfinite}

\newcommand{\SL}{\operatorname{SL}}
\subjclass[2010]{Primary: 43A07; Secondary: 20F65}

\thanks{
This research was supported in part by 
NSF grant DMS--1262019.
I would like to thank Matt Brin and Clinton Conley for
their help while preparing this article.
}

\author{Justin Tatch Moore}

\address{Department of Mathematics \\ Cornell University\\
Ithaca, NY 14853--4201 \\ USA}

\email{{\tt justin@math.cornell.edu}}
 
\begin{document}

%\begin{abstract}

%\end{abstract}

\maketitle

\noindent

\section{Introduction}

The notion of an \emph{amenable equivalence relation} was introduced
by Zimmer in the course of his analysis of orbit equivalence relations in ergodic theory (see \cite{KechrisMiller}).
More recently it played an important role in Monod's striking family of examples
of nonamenable groups which do not contain nonabelian free subgroups.
If $A$ is a subring of $\R$, define $H(A)$ to be the group of all piecewise $\PSL_2(A)$
homeomorphisms of the real projective line which fix the point at infinity.

\begin{thm}\cite{pw_proj_homeo} \label{Monod_thm}
If $A$ is any dense subring of $\R$, then $H(A)$ is nonamenable.
Moreover, if $f,g \in H(\R)$,
then either $\langle f,g \rangle$ is metabelian or else 
contains an infinite rank free abelian subgroup.
In particular, $H(\R)$ does not contain a nonabelian free subgroup.
\end{thm}
\noindent
Subsequently, Lodha and the author constructed a finitely presented nonamenable subgroup
of $H(\Z[1/\sqrt{2}])$ \cite{vN_fp}.
Lodha has since shown that the group of \cite{vN_fp} is moreover of type $F_\infty$ \cite{vN_Finfty}.

At least from a group-theoretic perspective, the most novel aspect of \cite{pw_proj_homeo}
was the use of Zimmer's notion of an amenable equivalence relation in the proof of the nonamenability of 
the groups $H(A)$.
The purpose of this article is to give a brief survey of the theory of
amenable and hyperfinite equivalence relations
and illustrate how it can be used to show that certain discrete groups are nonamenable.

The subject matter falls within the broader scope of what is sometimes called \emph{measurable group theory} ---
the study of groups through the analysis of their action on measure spaces.
This is in contrast with \emph{geometric group theory}, where the emphasis is on
actions which preserve an underlying geometry.
Measurable group theory is closely aligned with the ergodic theory and dynamics of discrete groups, probability,
and descriptive set theory.
Further reading can be found in \cite{DJK}, \cite{ICM:Gaboriau}, \cite{JKL}, \cite{KechrisMiller} and their references.
Further background on descriptive set theory can be found in \cite{set_theory:Kechris}.

This article is organized as follows.
After reviewing some background material and fixing some terminology, we will
present the definitions of amenable and hyperfinite equivalence relations in Section \ref{hyp_intro}.
This section will culminate with a theorem connecting these two apparently different notions.
Section \ref{examples_sec} will present several examples of nonamenable equivalence relations.
Section \ref{closure_sec}
will discuss the analogs of the closure properties of amenable groups in the setting of equivalence relations.
These played an important role in the isolation of the group in \cite{vN_fp}.

This article does not contain any new results, although Theorem \ref{abstract_GC} below is cast in a more
abstract way than in \cite{GhysCarriere}.
(It is also to the author's knowledge, the first account of this proof in English.)
The article's goal is to encourage the reader to pursue further reading in, e.g.,
\cite{KechrisMiller} which contains a much more complete treatment of the subject matter presented here.

\section{Preliminaries}

Before proceeding, we will fix some terminology.
In a few places we will refer to the continuity of extended real valued functions.
In this context, the neighborhoods of infinity are the co-bounded sets.
Recall that a \emph{Polish space} is a topological space which is separable and completely metrizable. 
The $\sigma$-algebra of Borel sets in a Polish space is said to be a \emph{standard Borel space}.
A function $f$ between standard Borel spaces $X$ and $Y$ is \emph{Borel} if preimages of Borel
sets are Borel.
This is equivalent to the graph of $f$ being a Borel subset of $X \times Y$.
It is well known that any two uncountable standard Borel spaces are isomorphic in the sense that there is 
a bijection $f$ between them such that $f$ and $f^{-1}$ are Borel.

A \emph{Borel measure} on a standard Borel space is a countably additive $\sigma$-finite measure
defined on its Borel sets.
Such a measure extends uniquely to the $\sigma$-algebra generated by the Borel sets and the
subsets of measure 0 Borel sets.
We will generally not distinguish between these measures but note here that \emph{measurable} will always refer
to the larger $\sigma$-algebra.

In this article we will write $(X,\mu)$ is a \emph{measured Polish space} to mean that $X$ is a Polish space
and that $\mu$ is a Borel measure on $X$.
If in addition the topology on $X$ is generated by the open sets of finite measure, then we say that
$(X,\mu)$ is \emph{locally finite}.
If $\Gamma$ is a topological group, then we will say that $\Gamma$ acts continuously on a measured
Polish space $(X,\mu)$
if:
\begin{itemize}

\item the map $(g,x) \mapsto g \cdot x$ is continuous and

\item the maps $g \mapsto \mu(g \cdot E)$ are continuous for each measurable $E \subseteq X$.

\end{itemize}
Notice that this is stronger than the assertion that $\Gamma$ acts continuously on the metric space $X$.

We note some useful facts about measured Polish spaces.

\begin{fact} \label{weak_Radon}
If $(X,\mu)$ is a measured Polish space and $E \subseteq X$ is measurable, then
$\mu(E)$ is the supremum of all $\mu(F)$ where $F$ is a closed subset of $E$.
\end{fact}

\begin{fact} \label{weak_density}
Suppose that $(X,\mu)$ is a locally finite measured Polish space.
If $E \subseteq X$ is measurable and has positive measure, then for every $\epsilon > 0$ there is an open set
$U \subseteq X$ such that $0 < \mu(U) < \infty$ and $\mu(E \cap U) > (1-\epsilon) \mu(U)$.
\end{fact}

We will also need the following proposition.

\begin{prop} \label{overlap}
Suppose that $(X,\mu)$ is a locally finite measured Polish space and $\Gamma$ is a metrizable group acting
continuously on $(X,\mu)$.
If $E \subseteq X$ is a measurable set of positive measure, then there is an open neighborhood $V$ of the identity
of $\Gamma$ and an $\epsilon > 0$ such that if $g$ is in $V$, then $\mu((g \cdot E) \cap E) > \epsilon$. 
\end{prop}

\begin{proof}
Let $E \subseteq X$ be given and let $U \subseteq X$ be an open set with
$$0 < \frac{3}{4} \mu(U) < \mu(E \cap U) < \mu(U) < \infty.$$
Set $\epsilon = \frac{1}{4} \mu(U)$.
Observe by our continuity assumption on the action,
we have that for every $x$ in $U$ there is an open $W_x \subseteq U$ containing $x$ and
a $\delta_x > 0$ such that if the distance from $g$ to the identity is less than $\delta_x$,
then $g \cdot W_x \subseteq U$.
Find a $\delta > 0$ such that
$$\mu(\{x \in U : \delta_x \geq \delta\}) > \frac{3}{4} \mu (U)$$
and define $W = \bigcup \{W_x : \delta_x \geq \delta\}$, observing that
$\mu(W) > \frac{3}{4} \mu (U)$.
In particular, $\mu(E \cap W) > \frac{1}{2} \mu (U)$.
By our assumption that $\Gamma$ acts continuously on $(X,\mu)$, there is an open set $V$
containing the identity such that every element of $V$ has distance less than $\delta$ to the identity and
$\mu(g \cdot (E \cap W)) > \frac{1}{2} \mu(U)$ whenever $g$ is in $V$.
Since $\mu( E \cap U ) > \frac{3}{4} \mu(U)$ and since $g \cdot (E \cap W) \subseteq (g \cdot E) \cap U$, it
follows that $\mu(E \cap (g \cdot E) \cap U) > \frac{1}{4} \mu (U) = \epsilon$.
\end{proof}

If $X$ is a standard Borel space,
an equivalence relation $E$ on $X$ is Borel if it is Borel as a subset of $X^2$.
A Borel equivalence relation is said to be \emph{countable} if every equivalence class is countable.
Notice that while this meaning conflicts with the literal interpretation of ``countable,'' there is never a
cause for confusion since for an equivalence relation to be countable as a set, it must have a countable
underlying set and in this context one is generally only interested in uncountable standard Borel spaces
(moreover the two notions coincide if the underlying space is countable).

The principal example of a countable Borel equivalence relation is as follows:
if $G$ is a countable discrete group acting by Borel automorphisms on a standard Borel space $X$,
then the orbit equivalence relation $E^G_X$ is a countable Borel equivalence relation.
That is, $(x,y)$ is in $E^G_X$ if and only if there is a $g$ in $G$ such that $g \cdot x = y$.
In fact all countable Borel equivalence relations arise in this way.

\begin{thm}\cite{FeldmanMoore}
If $E$ is a countable Borel equivalence relation on a standard Borel space, then there is a countable group
$G$ and a Borel action of $G$ on $X$ such that $E = E^G_X$.
\end{thm}

The advantage of working with equivalence relations
is in part that the notion of a countable Borel equivalence relation is much 
more flexible than that of a group.
For instance while subgroups give rise to subequivalence relations, the converse is not generally true.
A more sophisticated example of this is Theorem \ref{Monod_thm}:
orbit equivalence relations are used, in a sense, to transfer the nonamenability of $\PSL_2(A)$
to the group $H(A)$ even though these groups are quite unrelated from an algebraic perspective.

\section{Amenable and hyperfinite equivalence relations}

\label{hyp_intro}

Any action of a countable group on a standard Borel space gives rise to a countable Borel equivalence relation
and, conversely, any countable Borel equivalence relation can be generated as the orbit equivalence relation
of some group action.
The fundamental problem of this subject is to understand the extent to which properties of the
group which generated a countable Borel equivalence relation are reflected in properties of the equivalence
relation and vice versa.

Our focus in this article will be to develop
the properties of equivalence relations which are analogs
of the group-theoretic property of \emph{amenability}.
Roughly speaking,
the notion of an amenable equivalence relation has the property that every action of an amenable group gives
rise to an amenable equivalence relation and a group is amenable only when every orbit equivalence relation
is amenable.

Now to be more precise.
Suppose that $(X,\mu)$ is Borel measure on a standard Borel space and
$E$ is a countable Borel equivalence relation on $X$.
We say that $E$ is \emph{$\mu$-amenable} if there is a $\mu$-measurable assignment $x \mapsto \nu_x$ such that:
\begin{itemize}

\item each $\nu_x$ is a finitely additive probability measure on $X$ satisfying that
$\nu_x([x]_E) = 1$.

\item if $(x,y) \in E$, then $\nu_x = \nu_y$.

\end{itemize}
By \emph{measurable} we mean that if $A$ is any measurable subset of $X \times X$, then
\[
x \mapsto \nu_x(\{y \in X : (x,y) \in A\})
\]
is $\mu$-measurable.
While we will generally quantify \emph{amenable} with a measure, a Borel equivalence relation is
said to be \emph{amenable} if it is $\mu$-amenable with respect to every Borel measure on the underlying
standard Borel space.
 
While it is not apparent from the definition, it is true that every orbit equivalence relation
of a countable amenable group acting on standard Borel space is
necessarily $\mu$-amenable with respect to any Borel measure $\mu$
(it is interesting to note that this does not require any invariance properties of $\mu$ with respect
to the group action).
This will follow from Theorem \ref{amenEq_equivs} below.

Next we turn to a seemingly unrelated notion.
A countable Borel equivalence relation $E$ on a standard Borel space $X$
is \emph{hyperfinite} if $E$ is an increasing union
of Borel equivalence relations with finite equivalence classes.
A good example to keep in mind is that of \emph{eventual equality} on infinite binary sequences:
define $x =^* y$ if $x(k) = y(k)$ for all but finitely many $k$.
Notice that this is the union of the equivalence relations $=^n$ defined by
$x =^n y$ if $x(k) = y(k)$ for all $k \geq n$.

The following theorem gives a powerful criterion for verifying hyperfiniteness.

\begin{thm} \cite{DJK}
Suppose that $X$ is a standard Borel space and $f:X \to X$ is a Borel function which
such that $f^{-1}(x)$ is at most countable for each $x$.
The smallest equivalence relation $E$ satisfying that, for each $x \in X$, $(x,f(x)) \in E$ is hyperfinite.
\end{thm}

\begin{example} \cite{DJK}
Define an equivalence relation $E$ all infinite binary sequences by
$x E y$ if for some $m$ and $n$, $x(m+i) = y(n+i)$ for all $i > 0$.
This equivalence relation is called \emph{tail equivalence} and is generated by the shift map
$f:2^\omega \to 2^\omega$
given by $f(x)(i) = x(i+1)$. 
\end{example}

\begin{example} \cite{ConnesFeldmanWeiss} \label{PSL2(Z)}
Recall that the real projective line $P^1(\R)$ is the collection of all
lines in $\R^2$ passing through the origin.
Identify an element of $P^1(\R)$ with the $x$-coordinate of its intersection 
with the line $y=1$, adopting the convention that the line $y=0$ becomes identified with $\infty$.
This identification gives $P^1(\R)$ a natural compact metric topology --- it is in fact homeomorphic to
a circle.
An element of $\PSL_2(\R)$ can then be regarded as a fractional linear transformation
$t \mapsto \frac{at + b}{ct + d}$.
Define a map $\Phi : 2^{\omega} \to P^1(\R)$ by
\[
\Phi(\seq{1} x) = \phi (x) \qquad \qquad
\Phi(\seq{0} x) = - \phi (\sim x)
\]
where
\[
\phi(\seq{1} x) = 1+ \phi(x) \qquad \qquad \phi (\seq{0} x) = \frac{1}{1 + \frac{1}{\phi (x)}} 
\]
(here $\sim x$ denotes the bitwise complement of $x$).
This map preserves the cyclic order and is a quotient map from $2^\omega$ onto $P^1(\R)$.
The action of $\PSL_2(\Z)$ on $P^1(\R)$ naturally lifts to an action on $2^\omega$.
The corresponding orbit equivalence relation on $2^\omega$ is tail equivalence.
In particular, this orbit equivalence relation of $\PSL_2(\Z)$'s action on
$P^1(\R)$ is hyperfinite.
\end{example}

The following gives an important characterization of the hyperfinite equivalence relations.

\begin{thm} \cite{SlamanSteel} \cite{meas_dyn:Weiss} 
Every Borel action of $\Z$ on a standard Borel space generates a hyperfinite orbit equivalence relation.
Conversely, every hyperfinite Borel equivalence relation is the orbit equivalence relation of a Borel action of
$\Z$.
\end{thm}

While it is not obvious, it turns out that every hyperfinite Borel equivalence relation is in fact $\mu$-amenable
with respect to any Borel measure on the underlying space.
In fact, a natural weakening captures the notation of $\mu$-amenability exactly.
If $\mu$ is a Borel measure on $X$, then we say that $E$ is \emph{$\mu$-hyperfinite}
if there is a $\mu$-measure $0$ set $Y \subseteq X$ such that the restriction of $E$ to
$X \setminus Y$ is hyperfinite.

\begin{thm} (see \cite{KechrisMiller}) \label{amenEq_equivs}
Suppose that $X$ is a standard Borel space, $E$ is a countable Borel equivalence relation on $X$,
and $\mu$ is a Borel measure on $X$.
The following are equivalent:
\begin{enumerate}

\item $E$ is $\mu$-amenable;

\item $E$ is $\mu$-hyperfinite;

\item $E = E^G_X$ for some $\mu$-measurable action of an amenable group $G$ on $X$;

\item $E = E^{\Z}_X$ for some $\mu$-measurable action of $\Z$ on $X$.

\end{enumerate}
\end{thm}
\noindent
This is an amalgamation of several results stated in modern language.

It is worth noting that the extent to which the measure 0 sets can be omitted in the previous theorem is
a major open problem in descriptive set theory.

\begin{problem} \cite{DJK}
If $E_n$ $(n < \infty)$ is an increasing sequence of hyperfinite Borel equivalence relations on a standard Borel space,
is $\bigcup_{n=0}^\infty E_n$ hyperfinite?
\end{problem}

\begin{problem} \cite{meas_dyn:Weiss} \label{weiss_prob}
Is every orbit equivalence relation of a Borel action of a countable amenable group acting
on a standard
Borel space hyperfinite?
\end{problem}

In fact it was only recently that a positive solution to Problem \ref{weiss_prob}
was proved for the class of abelian groups \cite{abelian_hyp};
the strongest result at the time of this writing is \cite{loc_nil_hyp}.
While not directly related, Marks has recently demonstrated differences between the so-called
\emph{Borel context} and \emph{measure-theoretic context} \cite{det_Borel_comb}.

\section{Examples}

\label{examples_sec}

In this section we will consider a number of examples.
Perhaps the easiest way to generate nonamenable equivalence relations is through
actions of groups which preserve a probability measure.

\begin{thm} \cite{JKL}
Suppose that $G$ is a countable group, $(X,\mu)$ is a standard Borel space equipped with
a Borel probability measure, and $E$ is the orbit equivalence relation of a measure preserving 
action of $G$ which is free $\mu$-a.e..
The equivalence relation $E$ is $\mu$-amenable if and only if $G$ is amenable.
\end{thm}

The following are two typical --- but quite different --- examples of such actions.

\begin{example}
If $G$ is any countable group and $(X,\mu)$ is any standard probability space,
then $G$ acts by shift on $X^G$ as follows: $(g \cdot x)(h) = x(g^{-1} h)$.
This action preserves the product measure and, unless $\mu$ is a point-mass,
is free almost everywhere with respect to the product measure.
\end{example}

\begin{example}
The action of $\SL_2(\Z)$ on the torus $\T^2$ equipped with Lebesgue measure
is measure preserving and free $\lambda$-a.e..
Since $\SL_2(\Z)$ contains the free group on two generators, this orbit equivalence
relation is not $\lambda$-amenable.
\end{example}

It is interesting to contrast this previous example with that of the group
$\PSL_2(\Z)$ acting on 
$P^1(\R)$ (Example \ref{PSL2(Z)}), which is homeomorphic to the circle.
It is well known that $\PSL_2(\Z)$ contains a free subgroup (even one of finite index) and hence
is nonamenable.
On the other hand, we have seen above that the orbit equivalence relation induced on $P^1(\R)$ is
just tail equivalence on $2^\omega$ in disguise;
in particular it is hyperfinite and hence amenable.
Notice that, unlike the action of $\SL_2(\Z)$ on the torus,
there is no standard probability measure on $P^1(\R)$
which is preserved by the action of $\PSL_2(\Z)$.

It turns out, however, that \emph{dense} subgroups of $\PSL_2(\R)$ do produce
a nonamenable orbit equivalence relation when they act on $P^1(\R)$.

\begin{thm} \cite{GhysCarriere} \label{dense_to_free}
Every nondiscrete subgroup of $\PSL_2(\R)$ is either solvable or else contains a nondiscrete
free subgroup on two generators.
\end{thm}

\begin{thm} \cite{GhysCarriere} \label{free_to_nonamen}
If $\Gamma$ is a rank 2 free subgroup of a finite dimensional Lie group
$G$ and $\Gamma$ is nondiscrete, then the orbit equivalence relation of $\Gamma$'s action on
$G$ is nonamenable with respect to the Haar measure on $G$.
\end{thm}

In the case of $\Gamma = \PSL_2(\Z[1/2])$, it is not difficult to exhibit a
nondiscrete free subgroup explicitly.

\begin{example}
The matrices
\[
\alpha = \begin{pmatrix}
1/2 & -4 \\
1/4 &  0 \\
\end{pmatrix}
\qquad \qquad
\beta = \begin{pmatrix}
1/2 & -1/4 \\
4 & 0 \\
\end{pmatrix}
\]
generate a nondiscrete free subgroup of $\PSL_2(\Z[1/2])$.
In order to see this, first observe that the traces of these matrices are
$1/2$ and hence both matrices describe \emph{elliptic} transformations of the real projective line
(i.e. there are no fixed points).
Since any elliptic element of $\PSL_2(\R)$ of infinite order generates a nondiscrete subgroup,
it suffices to show that the above matrices generate a free group.

Define $X$ to be the set of all rational numbers in $P^1(\R) = \R \cup \{\infty\}$ which can be represented by a fraction with an odd denominator
and let $Y$ denote the remaining rational numbers in $P^1(\R)$.
By the \emph{Ping-Pong Lemma} (see, e.g., \cite{topics_geo_grp}), it suffices to show that if $n \ne 0$ is an integer, then
$\alpha^n Y \subseteq X$ and $\beta^n X \subseteq Y$.
Let $X_0$ consist of those elements of $X$ which can be represented by a fraction
of the form $(4p+2)/q$ where $q$ is odd.
Notice that $\alpha (X_0 \cup Y) \subseteq X_0$ and that $X_0$ and $Y$ are disjoint.
It follows that $\alpha^n Y \subseteq X$ whenever $n$ is a nonzero integer.
Similarly, $\beta^n X \subseteq Y$.
\end{example}

Theorem \ref{dense_to_free} was generalized considerably by the following result.

\begin{thm} \cite{dense_free_subgrp}
If $\Gamma$ is a dense subgroup of a connected semi-simple real Lie group,
then $\Gamma$ contains a dense free subgroup of rank 2.
\end{thm}

The following theorem is a generalization of Theorem \ref{free_to_nonamen}, although
the argument closely follows that of \cite{GhysCarriere}.

\begin{thm} \label{abstract_GC}
Suppose that $(X,\mu)$ is a locally finite measured Polish space.
If $\Gamma = \langle a,b \rangle$ is a
free nondiscrete metrizable group which is acting freely and continuously on $(X,\mu)$,
then the orbit equivalence relation is not $\mu$-amenable.
\end{thm}

\begin{proof}
Suppose for contradiction that $E^\Gamma_X$ is $\mu$-amenable and fix a $\mu$-measurable
assignment $x \mapsto \nu_x$ such that:
\begin{itemize}

\item for each $x$, $\nu_x$ is a finitely additive probability measure supported on the orbit of
$x$;

\item if $x$ and $y$ are in the same orbit, then $\nu_x = \nu_y$.

\end{itemize}
For $u \in \{a,b\}$, define $\Gamma_u$ to be all those elements of $\Gamma$ which are representable
by a reduced word beginning with $u$ and ending with $u^{-1}$.
Observe that if $g$ is a nonidentity element of $\Gamma$, then there is a $h$ in
$\{a,ab,ab^{-1}\}$ such that $h g h^{-1}$ is in $\Gamma_a$.
Thus there is an $h \in \{e,a,ab,ab^{-1}\}$ and a $\Gamma' \subseteq \Gamma$ which accumulates to
the identity such that $h \Gamma' h^{-1} \subseteq \Gamma_a$.
Since conjugation is continuous, it follows that $\Gamma_a$ also accumulates to the identity.
Furthermore, $b \Gamma_a b^{-1} \subseteq \Gamma_b$ and thus $\Gamma_b$ accumulates to the identity as well.

Since the action of $\Gamma$ on $X$ is free,
for each $x,y \in X$ which lie in the same orbit, there is a unique
$\gamma = \gamma(x,y)$ in $\Gamma$ such that $x = \gamma \cdot y$.
Notice that $\gamma(g\cdot x,y) = g^{-1} \gamma (x,y)$.
Define $\phi:X \to [0,1]$ by letting $\phi(x) = \nu_x(A_x)$ where
$A_x$ is the set of those $y$ in the orbit of $x$ such that the reduced word representing $\gamma(x,y)$ begins
with $a$ or $a^{-1}$.
Observe that for any $x$ in $X$ and $g$ in $\Gamma_b$, $A_{x} \cap A_{g \cdot x} = \emptyset$.
Hence if $x$ is in $X$ and $\phi(x) > 1/2$, then $\phi(g \cdot x) < 1/2$ whenever $g$ is in $\Gamma_b$.
Similarly, if $\phi(x) < 1/2$ and $g$ is in $\Gamma_a$, then $\phi(g \cdot x) > 1/2$.
Furthermore, for all $x$ and $g \in \Gamma_a$, the sets $A_{x}$, $A_{b g b^{-1} \cdot x}$ and
$A_{b^2 g b^{-2} \cdot x}$ are pairwise disjoint and
consequently $0 \leq \phi(x) + \phi(b g b^{-1} \cdot x) + \phi(b^2 g b^{-2} \cdot x) \leq 1$.

I next claim that $Y = \{x \in X : \phi(x) \ne 1/2\}$ has positive measure with respect to $\mu$.
Suppose not.
Using our assumption that $\Gamma$ acts continuously on $(X,\mu)$,
find an open neighborhood $V$ of the identity
such that if $g$ is in $V$,
then $Y$, $b g^{-1} b^{-1} \cdot Y$, and $b^2 g^{-1} b^{-2} \cdot Y$ have total measure less than that of $X$.
Now if $x$ is outside these sets and \(g \in V \cap \Gamma_a\),
we have that $\phi(x)$, $\phi(b g b^{-1} \cdot x)$, and $\phi(b^2 g b^{-2} \cdot x)$
are each $1/2$, contradicting that there sum is at most 1.
It must therefore be that $Y$ has positive measure.

Let $Y_a = \{y \in Y : \phi(y) >  1/2\}$ and $Y_b = \{y \in Y : \phi(y) < 1/2\}$.
Since $Y$ has positive measure, either $Y_a$ or $Y_b$ have positive measure.
If $Y_a$ has positive measure, then by Proposition \ref{overlap} there is a $g$
in $\Gamma_b$ such that $(g \cdot Y_a) \cap Y_a$ has positive measure and in particular
is nonempty.
This contradicts our observation that if $\phi(y) > 1/2$ and $g$ is in $\Gamma_b$, then
$\phi(g \cdot y) < 1/2$.
Similarly, if $Y_b$ has positive measure, one obtains a contradiction by finding
a $g$ in $\Gamma_a$ such that $g \cdot Y_b$ intersects $Y_b$.
It must be, therefore, that the orbit equivalence relation is nonamenable.
\end{proof}

We finish this section with a simple but powerful observation of Monod.
\begin{example} \cite{pw_proj_homeo}
If $A$ is a countable dense subring of $\R$, let $H(A)$ denote the group consisting of all
orientation preserving homeomorphisms of $P^1(\R)$ which fix the point at infinity and which
are piecewise $\PSL_2(A)$.
Suppose that $\alpha$ is in $\PSL_2(A)$ and that $\alpha$ does not fix $\infty$.
As a fractional linear transformation, the graph of $\alpha$ is a hyperbola.
If $r \in \R$ is sufficiently large in magnitude, then
$\alpha(t) = t + r$ has two solutions $a < b$; set
\[
\alpha_r(t) =
\begin{cases}
\alpha(t) & \textrm{ if } a < t < b \\
t+r & \textrm{ otherwise}.
\end{cases}
\]
If $r$ is moreover an integer,
then $\alpha_r$ is in $H(A)$.
It follows that the restriction of the orbit equivalence relation of $\PSL_2(A)$ to $\R$
coincides with the corresponding restriction of the orbit equivalence relation of $H(A)$'s action
on $\R$.
Thus, by Theorem \ref{amenEq_equivs} and the results of \cite{GhysCarriere} mentioned above,
$H(A)$ is nonamenable whenever $A$ is a dense subring of $\R$.
\end{example}

\begin{example} \cite{vN_fp}
Let $P(\Z)$ denote the subgroup of $H(\Z)$ consisting of those elements which have
a continuous derivative.
By unpublished work of Thurston, $P(\Z)$ is isomorphic to Richard Thompson's group $F$.
In fact
\[
\alpha(t) = t+1 \qquad \qquad
\beta(t) = \begin{cases}
t & \textrm{ if } t \leq 0 \\
\frac{t}{1-t} & \textrm{ if } 0 \leq t \leq \frac{1}{2} \\
3- \frac{1}{t} & \textrm{ if } \frac{1}{2} \leq t \leq 1 \\
t+1 & \textrm{ if } 1 \leq t
\end{cases}
\]
is the standard set of generators with respect to the usual finite presentation of $F$
(see \cite{vN_fp}).
It is not difficult to see that the orbit equivalence relation of $P(\Z)$'s action on $P^1(\R)$ coincides
with that of $\PSL_2(\Z)$ except for the point at infinity.
Since $\PSL_2(\Z) \cup \{t \mapsto t + 1/2\}$ generates $\PSL_2(\Z[1/2])$,
it follows that $\langle t \mapsto t/2, \beta \rangle$ is nonamenable.
\end{example}

The previous example is less relevant to the amenability problem for $F$, however, than
it might initially appear.
For instance, Lodha has shown that if $\Gamma$ is any subgroup of $H(\R)$ which is isomorphic to
$F$, then the orbit equivalence relation of $\Gamma$'s action on $\R$ is $\lambda$-amenable
where $\lambda$ is Lebesgue measure.
It is unclear whether this is true if $\Gamma$ is only assumed to be a subgroup of
the homeomorphism group of $\R$.

\section{Closure properties of amenable equivalence relations}

\label{closure_sec}

One of the most basic facts about amenable groups is that they are closed under taking subgroups, extensions,
and directed unions.
These operations have their analogs in the setting of countable Borel equivalence relations as well.
Notice that if $H \leq G$, then $E^H_X \subseteq E^G_X$ whenever $G$ acts on a standard Borel space.
Also, if $G$ is an increasing union of a sequence of subgroups $G_n$ $(n < \infty)$, then
$E^G_X = \bigcup_n E^{G_n}_X$.
Since the property of being $\mu$-hyperfinite is clearly inherited to subequivalence relations, we have
the following corollary of Theorem \ref{amenEq_equivs}.
 
\begin{prop}
If $E$ is a subequivalence relation of a $\mu$-amenable equivalence relation
is $\mu$-amenable.
\end{prop}

While there is no natural notion of extension in the setting of equivalence relations, 
it is easy to formulate what is meant by a product of equivalence relations.
It is straightforward to verify the following analog of the closure of the class of amenable groups under
taking products.

\begin{prop}
Products of $\mu$-amenable equivalence relations are amenable with respect to the corresponding product measure.
\end{prop}

The following is the corresponding analog of the amenability of increasing unions of amenable groups.

\begin{thm} \cite{MPT1:Dye} \cite{non-sing_trans_meas}\label{dye_krieger}
Suppose that $E_n$ $(n < \infty)$ is an increasing sequence of countable Borel equivalence relations on 
a standard Borel space $X$.
If $\mu$ is a standard measure on $X$ and each $E_n$ is $\mu$-hyperfinite, then $\bigcup_n E_n$ is
$\mu$-hyperfinite.
\end{thm}

The power of the closure properties mentioned in this section is that they afford some flexibility which
has no analog in the algebraic setting.
For instance, while the equivalence relations $E_n$ in the previous theorem are required to be nested,
they need not come from a nested sequence of groups.
It is also sometimes fruitful to generate equivalence relations with partial homeomorphisms rather than
full automorphisms of an underlying space.

\begin{example}
Consider the following homeomorphisms of $P^1(\R) =\R \cup \{\infty\}$:
$\alpha(t) = t+1/2$ and $\beta(t) = -1/t$.
For $0 < r < \infty$, define $\alpha_r$ to be the restriction of $\alpha$ to
$[-r,-1/r]$.
Let $E$ be the equivalence relation generated by
$\alpha$ and $\beta$ and
$E_r$ be the equivalence relation generated by $\beta$ and $\alpha_r$
(i.e. $E_r$ is the smallest equivalence relation such that for all $t$,
$(t,t+1/2) \in E_r$ and if additionally $-r \leq t \leq -1/r$, then
$(t,-1/t) \in E_r$).
Notice that if $0 < r < s < \infty$, then $E_r \subseteq E_s \subseteq E$ and
that $E = \bigcup_{r > 0} E_r$.
Thus there is an $r$ such that $0 < r < \infty$ such that
$E_r$ is nonamenable.
In fact a more careful analysis reveals that $E_{4} = E$, although this is
not relevant for the point we wish to illustrate here.
\end{example}

\begin{example}
Suppose that $\alpha$ and $\beta$ are homeomorphisms
such that the action of $\langle \alpha, \beta\rangle$ on $\R$
generates a nonamenable equivalence relation.
Suppose further that $\alpha_n$ $(n < \infty)$ and $\beta_n$ $(n < \infty)$
are sequences of homeomorphisms such that for all but countably many $t$,
$\alpha_n(t) = \alpha(t)$ and $\beta_n(t) = \beta(t)$ holds for all but finitely many $n$.
It follows that there exists an $n$ such that the action of
$\langle \alpha_n,\beta_n \rangle$ on $\R$
generates a nonamenable equivalence relation.
To see this, let $X_n$ denote the set of all $t$ in $\R$ such that
for all $k \geq n$, $\alpha_k(t) = \alpha(t)$ and $\beta_k(t) = \beta(t)$.
Define $E_n$ to be the equivalence relation generated by the restrictions of
$\alpha_n$ and $\beta_n$ to $X_n$.
It follows that $E_n$ $(n < \infty)$ is an increasing sequence of countable 
Borel equivalence relations which, off a countable subset of $\R$, unions to
the equivalence relation generated by $\alpha$ and $\beta$.
The claim now follows from Theorem \ref{dye_krieger}.
\end{example}

%\bibliography{../global}
%\bibliographystyle{plain}
%\end{document}

\def\Dbar{\leavevmode\lower.6ex\hbox to 0pt{\hskip-.23ex \accent"16\hss}D}

\end{document}